\documentclass[10pt,a4paper,oneside]{amsart}
\usepackage{graphicx}
\usepackage{cases}
\usepackage{xcolor}
\usepackage[pdfencoding=auto]{hyperref}
\newtheorem{theorem}{Theorem}[section]
\newtheorem{corollary}[theorem]{Corollary}
\newtheorem{proposition}[theorem]{Proposition}
\newtheorem{lemma}[theorem]{Lemma}

\theoremstyle{definition}
\newtheorem{definition}{Definition}
\newtheorem{example}{Example}

\theoremstyle{remark}
\newtheorem{remark}[theorem]{Remark}

\numberwithin{equation}{section}

\newcommand{\dF}{\tilde{F}}

\usepackage{float}
\usepackage{ulem} 
\usepackage{multirow} 
\begin{document}
\title{The deranged Bell numbers}
\author{BELBACHIR Hac\`ene}
\address{}
\curraddr{}
\email{}
\thanks{}

\author{DJEMMADA Yahia}
\address{}
\curraddr{}
\email{}
\thanks{}

\author{N\'EMETH L\`aszl\`{o}}
\address{}
\curraddr{}
\email{}
\thanks{}

\date{}

\begin{abstract}
    It is known that the ordered Bell numbers count all the ordered partitions of the set $[n]=\{1,2,\dots,n\}$. 
    In this paper, we introduce the deranged Bell numbers that count the total number of deranged partitions of $[n]$.
    We first study the classical properties of these numbers (generating function, explicit formula, convolutions, etc.), 
    we then present an asymptotic behavior of the deranged Bell numbers. Finally, we give some brief results for their $r$-versions.
\end{abstract}
\maketitle
\section{Introduction}
A \textit{permutation} $\sigma$ of a finite set $[n] := \{1,2, \dots, n\}$ is a rearrangement (linear
ordering) of  the elements of $[n]$, and we denote it by

$$\sigma([n])=\sigma(1)\sigma(2)\cdots\sigma(n).$$

A \textit{derangement} is a permutation $\sigma$ of $[n]$ that verifies $\sigma(i) \neq i$ for all $ (1 \leq i \leq n)$ (fixed-point-free permutation). The derangement number $d_n$ denotes the number of all derangements of the set $[n]$. A simple combinatorial approach yields the two recursions for $d_n$ (see for instance \cite{Lo})
\begin{equation*} \label{drec1}
    d_{n}=(n-1)(d_{n-1}+d_{n-2}) \quad ( n \geq 2)
\end{equation*} 
and
\begin{equation*}\label{drec2}
    d_{n}=n d_{n-1}+(-1)^n \quad ( n \geq 1),
\end{equation*}
with the first values $d_0=1$ and $d_1=0$.

The derangement number satisfies the explicit expression (see \cite{Bon})
\begin{equation*} \label{dexp}
    d_n=n!\sum_{i=0}^{n}\frac{(-1)^{i}}{i!}.    
\end{equation*}

The generating function of the sequence $d_n$ is given by
\begin{equation*}\label{degf}
    \mathcal{D}(t)=\sum_{n \geq 0} d_n \frac{t^n}{n!}=\frac{e^{-t}}{1-t}.    
\end{equation*}
The first few values of $d_n$ are
$$(d_n)_{n \geq 0}=\{1, 0, 1, 2, 9, 44, 265, 1854, 14833, 133496, 1334961, \dots\}.$$

For more details about derangement numbers we refer readers to \cite{Ch,Lo,Wa} and the references therein.

A \textit{partition} of a set $[n] := \{1, 2, \dots ,n\}$ is a distribution  of their elements to $k$ non-empty disjoint subsets $B_1|B_2|\dots|B_k$ called \textit{blocks}.
We assume that the blocks are arranged in ascending order according to their minimum elements ($\min B_1<\min B_2<\cdots<\min B_k$).

It is well-known that the \textit{Stirling numbers of the second kind}, denoted ${n \brace k}$, count the partitions' number of the set $[n]$ into $k$ non-empty blocks. The numbers ${n \brace k}$ satisfy the recurrence
\begin{equation*}\label{Stirec}
{n \brace k}={n -1\brace k-1}+k{n-1 \brace k} \qquad ( 1\leq k \leq n),
\end{equation*}
with ${n \brace 0}=\delta_{n,0}$ (Kronecker delta) and ${n \brace k}=0 \quad (k>n)$.

An \textit{ordered partition} $\psi$ of a set $[n]$ is a permutation $\sigma$ of the partition $B_1|B_2|\dots|B_k$, in other words, we consider all the orders of the blocks, 

$$\psi([n])=B_{\sigma(1)}|B_{\sigma(2)}|\dots|B_{\sigma(k)}.$$

For notation, throughout this paper we represent the elements of the same block by adjacent numbers and we separate the blocks by bars $"|"$.
\begin{example}
    The partitions of the set $[3]=\{1,2,3\}$ are:
    $${123};{1|23};{12|3};{13|2};{1|2|3},$$
and its ordered partitions are the permutations of all the partitions above:
$${123};{1|23};{23|1};{12|3};{3|12};{13|2};{2|13};{1|2|3};{1|3|2};{2|1|3};{2|3|1};{3|1|2};{3|2|1}.$$
\end{example}

The total number of the ordered partitions of the set $[n]$ is known as the \textit{ordered Bell number} or \textit{Fubini number} \cite{Com,Man}, denoted by $F_n$, which is given by
\begin{equation*}\label{Fubdef}
    F_n=\sum_{k=0}^n k! {n \brace k}.
\end{equation*}

The first few values of the ordered Bell numbers are
$$(F_n)_{n \geq 0}=\{1, 1, 3, 13, 75, 541, 4683, 47293, 545835, 7087261, 102247563, \dots\}.$$

The explicit formula for the Stirling number of the second kind
\begin{equation*}\label{Stiexp}
    {n \brace k}=\frac{1}{k!}\sum_{j=0}^{k}(-1)^{k-j}{k \choose j} j^n
\end{equation*} 
follows the explicit formula for the ordered Bell number
\begin{equation*}\label{Fubexp}
    F_n=\sum_{k=0}^n\sum_{j=0}^{k}(-1)^{k-j}{k \choose j} j^n.
\end{equation*}

The exponential generating function for $F_n$ is given by 
\begin{equation}\label{Fubegf}
    \mathcal{F}(t)=\sum_{n \geq 0}F_n\frac{t^n}{n!}=\frac{1}{2-e^t}.
\end{equation}
We note that if the order of the blocks does not matter, then the total number of partitions of a set $[n]$ is given by \text{Bell numbers}
\begin{equation*}\label{Beldef}
    B_n=\sum_{k=0}^n {n \brace k}.
\end{equation*}
The exponential generating function for $B_n$ is
\begin{equation}\label{Belegf}
    \mathcal{B}(t)=\sum_{n \geq 0}B_n\frac{t^n}{n!}=e^{e^t-1}.
\end{equation}

Most of the previous works focused on the some restrictions and generalizations of Stirling number of second kind to introduce new classes of ordered Bell number (see \cite{Cai} and the references given there).

The aim of our paper is to introduce and study a new classes of ordered partitions numbers by taking into account the derangement of blocks (or permutations without fixed blocks).

\section{The deranged Bell numbers }
In this section, we introduce the notion of deranged partition and we study the deranged Bell numbers.
\begin{definition}
    A \textit{deranged partition} $\tilde{\psi}$ of the set $[n]$ is a derangement $\tilde{\sigma}$ of the partition $B_1|B_2|\dots|B_k$, i.e.,
     $$\tilde{\psi}([n])=B_{\tilde{\sigma}(1)}|B_{\tilde{\sigma}(2)}|\dots|B_{\tilde{\sigma}(k)}$$
    such that $B_{\tilde{\sigma}(i)} \neq B_{i}$ for all $(1 \leq i \leq k)$.
\end{definition}
\begin{definition}
Let $\dF_n$ be the \textit{deranged Bell number} which counts the total number of the deranged partitions of the set $[n]$.
\end{definition}
\begin{proposition} For all $n \geq 0$ we have that
    \begin{equation}\label{defdF}
        \dF_n=\sum_{k=0}^n d_k {n \brace k}.
    \end{equation}
\end{proposition}  
\begin{proof} 
    Since $\displaystyle {n \brace k}$ counts the number of partitions of $[n]$ into $k$ blocks, then the number of deranged partitions of $[n]$ having $k$ blocks is $ \displaystyle d_k {n \brace k}$ (derangement of blocks).
    Therefore the $n^{\text{th}}$ deranged Bell number is $\displaystyle \dF_n=\sum_{k=0}^n d_k {n \brace k}.$ 
\end{proof}

Here are the first few values of $\dF_n$:
$$(\dF_n)_{n\geq 0}=\{1, 0, 1, 5, 28, 199, 1721, 17394, 200803, 2607301, 37614922,\dots\}.$$
In Tables \ref{DP3} and \ref{DP4}, we give few examples of the deranged permutations.
\begin{table}[H] \renewcommand{\arraystretch}{1.2}
\begin{center}
	\begin{tabular}{|c|c|c|c|} 
		\hline 
		Set         &Partition $b_1|b_2|\cdots |b_k$  &Deranged\ partitions& $\dF_n$ \\
		\hline
		\multirow{4}{*}{}& 123                           &      $\emptyset$            &\\
		\cline{2-3}
		& $1|23$                          &$23|1 $               &       \\
		\{1,2,3\}   & $12|3$                          &$3|12$                &    5   \\
		& $13|2$                          &$2|13$                &       \\
		\cline{2-3}
		& $1|2|3$                         &$2|3|1 \ \ 3|1|2$    &       \\
		\hline
	\end{tabular}
\end{center}
	\caption{Deranged partitions of the set $[3]$.}
	\label{DP3}
\end{table}

    \begin{table}[H]  \renewcommand{\arraystretch}{1.2}
        $$\begin{array}{|c|c|c|c|}
    \hline
    \text{Set}         &\text{Partition}\ b_1|b_2|\cdots |b_k  &\text{Deranged partitions}& \dF_n  \\
    \hline            
    \multirow{5}{*}{}   &1234       &            \emptyset                 &\\
    \cline{2-3}
                        &1|234      &234|1                          &\\
                        &12|34      &34|12                          &\\
                        &134|2      &2|134                          &\\
                        &123|4      &4|123                          &\\
                        &14|23      &23|14                          &\\
                        &124|3      &3|124                          &\\
                        &13|24      &24|13                          &\\
    \cline{2-3}
    \{1,2,3,4\}         &1|2|34     &2|34|1\ \ 34|1|2               &28\\
                        &1|23|4     &23|4|1\ \ 4|1|23               &\\
                        &1|24|3     &24|3|1\ \ 3|1|24               &\\
                        &12|3|4     &3|4|12\ \ 4|12|3               &\\
                        &13|2|4     &2|4|13\ \ 4|13|2               &\\
                        &14|2|3     &2|3|14\ \ 3|14|2               &\\
    \cline{2-3}
                        &           &2|1|4|3\ \ 2|3|4|1\ \ 2|4|1|3  &\\
                        &1|2|3|4    &3|1|4|2\ \ 3|4|1|2\ \ 3|4|2|1  &\\
                        &           & 4|1|2|3\ \ 4|3|1|2\ \ 4|3|2|1 &\\
    \hline
    \end{array}$$
    \caption{Deranged partitions of the set $[4]$.}
    \label{DP4}
    \end{table}

\section{Fundamental properties}
Here are the fundamental properties of the deranged Bell numbers.
\subsection{Exponential generating function}
\begin{theorem}\label{dFegf} The exponential generating function of deranged Bell numbers is given by
    \begin{equation*}
        \mathcal{\dF}(t)=\sum_{n \geq 0}\dF_n\frac{t^n}{n!}=\frac{e^{-(e^t-1)}}{2-e^t}.
    \end{equation*}
\end{theorem}
\begin{proof} Denote by $\mathcal{\dF}(t)$ the exponential generating function of the sequence $\dF_n$. From \eqref{defdF} we have 
    \begin{equation*}
            \mathcal{\dF}(t)=\sum_{n \geq 0}\sum_{k=0}^n d_k {n \brace k}\frac{t^n}{n!}
                            =\sum_{k \geq 0}d_k \sum_{n \geq 0} {n \brace k} \frac{t^n}{n!}
                            =\sum_{k \geq 0}d_k \frac{(e^t-1)^k}{k!}
                            =\frac{e^{-(e^t-1)}}{2-e^t}.
    \end{equation*}
    \end{proof}
    \subsection{Explicit formula}
    \begin{theorem} For any $ n \geq 0$, the sequence $\dF_n$ can be expressed explicitly as
        \begin{equation*}
            \dF_n=\sum_{k = 0}^n \sum_{i,j=0}^k\frac{(-1)^{k+i-j}}{i!}{k \choose j}j^n.
        \end{equation*}
    \end{theorem}
    \begin{proof}
        From the explicit formulas of Stirling numbers  of the second kind and derangement number we have
    
        \begin{equation*}
            \resizebox{1\textwidth}{!}{$
            \displaystyle
                \dF_n   =\sum_{k=0}^n d_k {n \brace k}=\sum_{k=0}^n k!\sum_{i=0}^{k}\frac{(-1)^i}{i!} \frac{1}{k!}\sum_{j=0}^{k}(-1)^{k-j}{k \choose j} j^n=\sum_{k = 0}^n  \sum_{i,j=0}^k\frac{(-1)^{k+i-j}}{i!}{k \choose j}j^n.
                $}
        \end{equation*}
    \end{proof}
    \subsection{Dobi\`nski's formula}
    One of the most important result for Bell number was established by Dobi\`nski \cite{Dob,Eps,Rot}, where he expressed $w_n$ in the infinite series form bellow
    \begin{equation*}
        w_n=\frac{1}{e}\sum_{k \geq 0}\frac{k^n}{k!}.
    \end{equation*}
    An analogue result for the ordered Bell number was given by Gross \cite{Gro} as
    \begin{equation*}
        F_n=\frac{1}{2}\sum_{k \geq 0}\frac{k^n}{2^k}.
    \end{equation*}
    The Dobi\`nski's formula for $\dF_n$ is established by our next theorem
    \begin{theorem} For any $n \geq 0$ we have
        \begin{equation*}
            \dF_n=\frac{e}{2}\sum_{j\geq 0}\sum_{k\geq j}  \frac{(-1)^j}{j!}\frac{k^n}{2^{k-j}}.            
        \end{equation*}
    \end{theorem}
    
    \begin{proof} From Theorem \ref{dFegf} it follows that
        \begin{equation*} 
            \begin{split}
                \sum_{n \geq 0} \dF_n\frac{t^n}{n!}&=\frac{e^{-\left(e^t-1\right)}}{1-\left(e^t-1\right)}=\frac{e}{2}\frac{ e^{-e^t}}{\left(1-\frac{e^t}{2}\right)}\\
                &=\frac{e}{2}\sum_{k\geq 0} \frac{(-1)^k  e^{kt}}{k!}\sum_{k\geq 0}\left(\frac{1}{2}\right)^k e^{kt}\\
                &=\frac{e}{2}\sum_{k\geq 0} \sum_{j=0}^k  \frac{(-1)^j e^{jt}}{j!}\left(\frac{1}{2}\right)^{k-j} e^{(k-j)t}\\
                &=\frac{e}{2}\sum_{n\geq 0}\sum_{k\geq 0} \sum_{j=0}^k  \frac{(-1)^j}{j!}\left(\frac{1}{2}\right)^{k-j}k^n \frac{t^n}{n!}\\
                &=\frac{e}{2}\sum_{n\geq 0} \sum_{j\geq0}\sum_{k\geq j}  \frac{(-1)^j}{j!}\left(\frac{1}{2}\right)^{k-j}k^n \frac{t^n}{n!},\\
            \end{split}
        \end{equation*}
        by comparing the coefficient of $\frac{t^n}{n!}$ we get
        $$\dF_n=\frac{e}{2}\sum_{j\geq 0}\sum_{k\geq j}  \frac{(-1)^j}{j!}\frac{k^n}{2^{k-j}}.$$
    \end{proof}
     \begin{remark}
         Dobi\`niski's formula is suitable to the computation of $w_n$, $F_n$ and $\dF_n$ for large $n$ values as Rota mentioned in \cite{Rot}.
     \end{remark}
    \section{Higher order derivatives and convolution formulas}
    Before giving our next result, we state the following lemma and proposition.
    \begin{lemma}For any $m \geq 1$, the $m^{\text{th}}$ derivatives of $\mathcal{F}(t)$ and $\frac{1}{\mathcal{W}(t)}$ are, respectively,
        \begin{equation}\label{Fm}
            \mathcal{F}^{(m)}(t)=\sum_{k=0}^m k! {m \brace k} e^{kt} \mathcal{F}^{k+1}(t)
        \end{equation}
        and
        \begin{equation}\label{Bm}
            \left(\frac{1}{\mathcal{W}(t)}\right)^{(m)}=\sum_{k=0}^m(-1)^k {m\brace k}\frac{e^{kt}}{\mathcal{W}(t)}.
        \end{equation}
    \end{lemma}
    \begin{proof}
        The proof of lemma proceeds by induction on $m$. 
        
        For $m=1$ it is easy to check from the generating functions of Fubini numbers \eqref{Fubegf} and Bell numbers \eqref{Belegf} that 
    
        \begin{equation}\label{m1F}
            \mathcal{F}'(t)=e^t\mathcal{F}^2(t)=\sum_{k=0}^1 k! {1 \brace k} e^{kt} \mathcal{F}^{k+1}(t) 
        \end{equation} 
        and 
        
        \begin{equation}\label{m1B}
        \left(\frac{1}{\mathcal{W}(t)}\right)'= \frac{-e^t}{\mathcal{W}(t)}=\sum_{k=0}^1(-1)^k {1\brace k}\frac{e^{kt}}{\mathcal{W}(t)}.
        \end{equation}
        Then from \eqref{m1F} and \eqref{m1B} the lemma is true for $m=1$.

        Now, assume the lemma holds for a fixed $m \geq 1 $, then
        \begin{equation*}
            \mathcal{F}^{(m)}(t)=\sum_{k=0}^m k! {m \brace k} e^{kt} \mathcal{F}^{k+1}(t)
        \end{equation*}
        and
        \begin{equation*}
            \left(\frac{1}{\mathcal{W}(t)}\right)^{(m)}=\sum_{k=0}^m(-1)^k {m\brace k}\frac{e^{kt}}{\mathcal{W}(t)}.
        \end{equation*}
        Now, we prove the statement for $m+1$. Thus,
        \begin{equation*}
            \begin{split}
                \mathcal{F}^{(m+1)}(t)&=\left(\sum_{k=0}^m k! {m \brace k} e^{kt} \mathcal{F}^{k+1}(t)\right)'\\
                                        &=\sum_{k=0}^m k! {m \brace k} \left[ke^{kt} \mathcal{F}^{k+1}(t)+(k+1)e^{kt} \mathcal{F}^{k}(t) \mathcal{F}'(t)\right]\\
                                        &=\sum_{k=0}^m k! {m \brace k} k e^{kt} \mathcal{F}^{k+1}(t)+\sum_{k=0}^m (k+1)! {m \brace k}e^{(k+1)t} \mathcal{F}^{k+2}(t)\\
                                        &=\sum_{k=0}^{m+1} k! {m \brace k} k e^{kt} \mathcal{F}^{k+1}(t)+\sum_{k=0}^{m+1} k! {m \brace k-1}e^{kt} \mathcal{F}^{k+1}(t)\\
                                        &=\sum_{k=0}^{m+1} k! \left(k{m \brace k}+{m \brace k-1} \right)e^{kt} \mathcal{F}^{k+1}(t)\\      
                                        &=\sum_{k=0}^{m+1} k! {m+1 \brace k} e^{kt} \mathcal{F}^{k+1}(t)
            \end{split}
        \end{equation*}
        and 
        \begin{equation*}
            \begin{split}
                \left(\frac{1}{\mathcal{W}(t)}\right)^{(m)}&=\left(\sum_{k=0}^m(-1)^k {m\brace k}\frac{e^{kt}}{\mathcal{W}(t)}\right)'\\
                                        &=\sum_{k=0}^m (-1)^k {m\brace k} \left[\frac{e^{kt}(k\mathcal{W}(t)-\mathcal{W}'(t))}{\mathcal{W}^2(t)}\right]\\
                                        &=\sum_{k=0}^m (-1)^k {m \brace k} k \frac{e^{kt}}{\mathcal{W}(t)}+\sum_{k=0}^m (-1)^{k+1} {m \brace k}\frac{e^{(k+1)t}}{\mathcal{W}(t)}\\
                                        &=\sum_{k=0}^{m+1} (-1)^k {m \brace k} k \frac{e^{kt}}{\mathcal{W}(t)}+\sum_{k=0}^{m+1} (-1)^{k} {m \brace k-1}\frac{e^{kt}}{\mathcal{W}(t)}\\
                                        &=\sum_{k=0}^{m+1} (-1)^k \left(k{m \brace k}+{m \brace k-1} \right)\frac{e^{kt}}{\mathcal{W}(t)}\\      
                                        &=\sum_{k=0}^{m+1} (-1)^k {m+1 \brace k}\frac{e^{kt}}{\mathcal{W}(t)}.
            \end{split}
        \end{equation*}
        Therefore, the assumption holds true for $m+1$, which complete the proof.
    \end{proof}
    We can now formulate the higher order derivative for $\mathcal{\dF}(t)$. First of all, we define the so-called $i^\text{th}$ falling factorial of $j$ by
    \begin{equation*}
	j^{\underline{i}}=\begin{cases}
			i(i-1)(i-2)\cdots(i-j+1),& \text{ if } i\geq0;\\
			1, & \text{ if } i=0. \end{cases}
	\end{equation*}
    \begin{theorem} For any $m \geq 1$ we have
        \begin{equation}\label{dFmder}
            \mathcal{\dF}^{(m)}(t)=\mathcal{\dF}(t)\sum_{k=0}^m {m \choose k}\sum_{i=0}^{k}\sum_{j=0}^{m-k}(-1)^{j}i!{k \brace i}{m-k \brace j}e^{(j+i)t}\mathcal{F}^{i}(t)
        \end{equation}
        or equivalently 
        \begin{equation}\label{dFmder2}
            \mathcal{\dF}^{(m)}(t)=\mathcal{\dF}(t)\sum_{i=0}^m\sum_{j=i}^{m}(-1)^{i+j}e^{jt}j^{\underline{i}}{m \brace j}\mathcal{F}^{i}(t),
        \end{equation}   
        where $\mathcal{\dF}^{(m)}(t)$ is the $m^{\text{th}}$ derivative of $\mathcal{\dF}(t)$.
    \end{theorem}
    
    \begin{proof}
        From the generating function (Theorem \ref{dFegf}), it is easy to observe that
        \begin{equation*}
            \mathcal{\dF}(t)=\frac{\mathcal{F}(t)}{\mathcal{W}(t)}.
        \end{equation*} 
    According to Leibniz's formula (see section $5.11$, exercise $4$ in \cite{Ap}) 
    $$ (f(t)g(t))^{(n)}=\sum_{k=0}^n {n \choose k}f^{(k)}(t)g^{(n-k)}(t)$$
    we have
    \begin{equation*}
    	\mathcal{\dF}^{(m)}(t)=\left(\frac{\mathcal{F}(t)}{\mathcal{W}(t)}\right)^{(m)}=\sum_{k=0}^m {m \choose k}\mathcal{F}^{(k)}(t)\left(\frac{1}{\mathcal{W}(t)}\right)^{(m-k)}.
    \end{equation*}
  And from the precedent Lemma (equations \eqref{Fm} and \eqref{Bm}), we get the identity \eqref{dFmder}.

        For the equivalent identity \eqref{dFmder2}, we use the equation \cite[page 120]{Qu}
        \begin{equation*}
            \sum_{k=i}^m {m \choose k}{k \brace i}{m-k \brace j-i}={j \choose i}{m \brace j},
        \end{equation*}
        a simple calculation gives the result.
    \end{proof}
    
    Let us give an important consequence of the preceding theorem.
    \begin{corollary}\label{dFBinConv} For any $ n \geq 0$, the $n^{\text{th}}$ deranged Bell number satisfies the following binomial convolution
        \begin{equation*}
            \dF_{n+1}=\sum_{i=0}^{n}\sum_{j=0}^{i-1} {n \choose i}{i \choose j}\dF_{j}F_{i-j}.
        \end{equation*}
    \end{corollary}
    \begin{proof}
        To deduce the result from \eqref{dFmder} for $m=1$ we have 
        \begin{equation*}
            \begin{split}
                \mathcal{\dF}'(t)&=e^t\mathcal{\dF}(t)\left(\mathcal{F}(t)-1\right)\\
                &=e^t\left(\sum_{ n\geq 0}\dF_n\frac{t^n}{n!}\sum_{ n\geq 0}F_n\frac{t^n}{n!}-\sum_{ n\geq 0}\dF_n\frac{t^n}{n!}\right)\\
                &=e^t\sum_{ n\geq 0}\left(\sum_{j=0}^{n-1}{n \choose j}\dF_{j} F_{n-j}\right)\frac{t^n}{n!}\\
                &=\sum_{n \geq 0}\sum_{i=0}^n\sum_{j=0}^{i-1}{n\choose i}{i\choose j} \dF_{j} F_{i-j}\frac{t^n}{n!}.
            \end{split}
        \end{equation*}   
        In another hand, it is easy to check that 
        $$\mathcal{\dF}'(t)=\sum_{n \geq 0} \dF_{n+1}\frac{t^n}{n!}.$$
        Therefore, the corollary holds true.
    \end{proof}
    As a more general result we have
    \begin{corollary}For any $ n \geq 0$, the $n^{\text{th}}$ deranged Bell number satisfies the following multinomial convolution
    \begin{equation*}
        \dF_{n+m}=\sum_{i=0}^{m}\sum_{j=i}^{m}\sum_{k_1+k_2+\cdots+k_{i+2}=n} {n \choose {k_1,k_2,\dots,k_{i+2}}}(-1)^{i+j}j^{\underline{i}}{m \brace j}j^{k_1}\dF_{k_2}\prod_{s=3}^{i+2}F_{k_s}.
    \end{equation*}
    \end{corollary}
    \begin{proof}
    By applying generalized Cauchy product rule on identity \eqref{dFmder2} and comparing the coefficients of $\frac{t^n}{n!}$ we get the convolution. 
    \end{proof}
     \begin{corollary}For all $n \geq 1$ we have
        \begin{equation*}
            \sum_{j=1}^n {n \choose j}\dF_{n-j}F_j=\sum_{j=0}^n(-1)^{n-j}{n \choose j}\dF_{j+1}.
        \end{equation*}
     \end{corollary} 
     \begin{proof}The result holds true by applying the well-known binomial inversion formula (see for example \cite[Corollary 3.38, p. 96]{Aig}) on Corollary \ref{dFBinConv}.
    \end{proof}

    \section{\texorpdfstring{Asymptotic behavior $\dF_n$}{Asymptotic behavior}}
    
    In this section, we are interested to obtaining the asymptotic behavior the deranged Bell numbers $\dF_n$.
    
    Finding an asymptotic behavior of a sequence $(a_n)_{n\geq 0}$ means to find a second sequence $b_n$ simple than $a_n$ which gives a good approximation of its values when $n$ is large.

    We will use the classical singularity analysis technic (see for instance \cite{Fla} and Chapter 5 of \cite{Wil}) to deduce the asymptotic behavior a sequence $a_n$ using the singularities of its generating function $\mathcal{A}(t)$.
    
    \begin{theorem} The asymptotic behavior $\dF_n$ is
        \begin{equation*}
            \frac{\dF_n}{n!} \sim \frac{1}{2 e \log^{n+1} (2)}+ O\left(6.3213)^{-n}\right), \qquad n \longrightarrow\infty .
        \end{equation*}
    \end{theorem}
    \begin{proof}
    We can summarize the singularity analysis technic in the following steps:
        \begin{itemize}
            \item Compute the singularities of $\mathcal{A}(t)$.
            \item Compute the dominant singularity $\chi_0$ (singularity of smallest modulus).
            \item Compute the residue of $A(t)$ at $\chi_0$ $$Res(A(t);t=\chi_0)=\lim_{t \to \chi_0}(t - \chi_0)\mathcal{A}(t).$$
            \item The generating function $\mathcal{A}(t)$ satisfies 
            $$\mathcal{A}(t) \sim \mathcal{W}(t)=\frac{Res(A(t);t=\chi_0)}{(t-\chi_0)}.$$ 
            \item By comparing Taylor series coefficients of 
            $$\mathcal{A}(t)=\sum_{n\geq 0}a_n\frac{t^n}{n!}\quad \text{and}\quad \mathcal{C}(t)=\sum_{n\geq 0}c_n\frac{t^n}{n!}.$$ 
            We get the asymptotic behavior $a_n$ when $n$ is big enough given by 
             $$a_n \sim c_n+\mathcal{O}(\rho^{-n}),\qquad n \longrightarrow\infty,$$
             where $\rho$ is the modulus of the next-smallest modulus singularity.
           \end{itemize}
    Now, applying the previous steps on the generating function $\mathcal{\dF}(t)=\frac{e^{-(e^t-1)}}{2-e^t}$, the singularities of $\mathcal{\dF}(t)$ are $\chi_k=log(2)+2ki\pi$.

    The dominant singularity is $\chi_0=log(2)$ and the residue at this point is 
    $$Res(\mathcal{\dF}(t),t=\chi_0)=\lim_{t \to \chi_0}(t-\chi_0)\mathcal{\dF}(t)=-\frac{1}{2 e}.$$
    Thus 
    $$\mathcal{\dF}(t)\sim \frac{1}{2 e (\log (2)-t)}=\frac{1}{2 e }\sum_{n\geq 0}\frac{t^n}{\log^{n+1} (2)}.$$
    Therefore the asymptotic behavior $\dF_n$ is
    $$\frac{\dF_n}{n!} \sim \frac{1}{2 e \log^{n+1} (2)} +\mathcal{O}(\rho^{-n}),\qquad n \longrightarrow\infty,$$
    where $\rho=\sqrt{\log^2(2)+(2i\pi)^2}\simeq6.3213.$
\end{proof}
\section{The ${r}$-deranged Bell number}
The $r$-version of special numbers is a common natural extension in enumerative combinatorics, see, for example, the $r$-Stirling numbers \cite{Bro}, $r$-Bell numbers \cite{Mez}, $r$-Fubini numbers \cite{Cai}, $r$-derangement numbers \cite{Wa}.

\bigskip
The motivation of this section came from two earlier researches:
\begin{itemize}
    \item The $r$-Stirling numbers introduced by Border \cite{Bro}. The $r$-Stirling numbers of the second kind, denoted ${n \brace k}_r$, count the number of partitions $\pi$ of the set $[n]$ having exactly 
    $k$ blocks such that the $r$ first elements $1,2,\dots,r$ must be in distinct blocks.

    The $r$-Stirling numbers of the second kind kind have the following generating function \cite{Bro}
    $$\sum_{n \geq 0 }{ n+r \brace k+r}_r \frac{t^n}{n!}=\frac{e^{rt}(e^z-1)^k}{k!},$$
    
     and their explicit formula \cite{Mez} is
    $${ n+r \brace k+r}_r =\frac{1}{k!}\sum_{j=0}^k (-1)^{k-j}{k \choose j}(j+r)^n.$$
    
    \item The $r$-derangement numbers introduced by Wang et al \cite{Wa}. The $r$-derangement numbers, denoted $d_{n,r}$, count the number of permutations of the set $[n+r]$ having no fixed points such that the $r$ first elements $1,2,\dots,r$ must be in distinct cycles. 
    The exponential generating function for $r$-derangement numbers is
    $$\sum_{n \geq 0 }d_{n,r} \frac{t^n}{n!}=\frac{t^{r}e^{-t}}{(1-t)^{r+1}},$$
    and their explicit formula is 
    $$d_{n,r}=\sum_{i=r}^n{i \choose r}\frac{n!}{(n-i)!}(-1)^{n-i}, n\geq r.$$
\end{itemize}
Now, it's natural to define the $r$-deranged Bell numbers as

\begin{definition}
    An \textit{$r$-deranged partition} $\tilde{\Psi}$ of the set $[n+r]$ is an $r$-derangement $\tilde{\sigma}$ of the set of partitions $B_1|B_2|\dots|B_r|B_{r+1}|\dots|B_{k+r}$, i.e.,
     $$\tilde{\Psi}([n+r])=B_{\tilde{\sigma}(1)}|B_{\tilde{\sigma}(2)}|\dots|B_{\tilde{\sigma}(r)}|B_{\tilde{\sigma}(r+1)}|\dots|B_{\tilde{\sigma}(k+r)}$$
    such that $B_{\tilde{\sigma}(i)} \neq B_{i}$ for all $(1 \leq i \leq k+r)$.
\end{definition}
\begin{definition}
The \textit{$r$-deranged Bell numbers}, denoted $\dF_{n,r}$, count the total number of the deranged partitions of the set $[n+r]$.

It's clear that, for all positive integers, $n$, $k$ and $r$ with $( r \leq k \leq n)$ , we have
\begin{equation*}\label{dFnr}
    \dF_{n,r}=\sum_{k=0}^{n}d_{k,r}{n+r \brace k+r }.
\end{equation*}
\end{definition}

\setcounter{subsection}{0}
\subsection{Main properties of the $r$-deranged Bell numbers}
Let us give briefly the main properties of the $r$-deranged Bell numbers. The proofs are similar to the proves of previous results, so we  leave the verifications to the readers.

\begin{itemize}
    \item For all positive integers $n$, $k$ and $r$, the exponential generating function of $\dF_{n,r}$ is 
    \begin{equation*}
        \sum_{n \geq 0}\dF_{n,r}\frac{t^n}{n!}=\frac{\left(e^{t}(e^t-1)\right)^r e^{-\left(e^t-1\right)}}{(2-e^t)^{r+1}}.
    \end{equation*}  
   \item For all positive integers $n$, $k$ and $r$, the $r$-deranged Bell numbers satisfy 
   \begin{equation*}
    \dF_{n,r}=\sum_{k=0}^n\sum_{i=r}^k\sum_{j=0}^k {i \choose r}{k \choose j}\frac{(-1)^{k+i-j}}{i!}(j+r)^n,\quad n \geq r.
   \end{equation*}
   \item The $r$-deranged Bell numbers have the following Dobi\`nski-like formula
    \begin{equation*}
        \dF_{n,r}=\frac{e}{2^{r+1}}\sum_{j \geq 0}\sum_{k \geq j}\sum_{i=0}^r\frac{(-1)^{k+i-j}}{2^j(k-j)!}{r\choose i}{j+r\choose j}(2r+k-i)^n.
    \end{equation*}
\end{itemize}
Here are the first few $r$-deranged Bell numbers.
\begin{table}[H]
  \resizebox*{1\textwidth}{!}{$$$
\begin{array}{|c|cccccccc|}
 \hline
 &n=0&n=1&n=2&n=3&n=4&n=5&n=6&n=7\\
 \hline
 r=0&1 & 0 & 1 & 5 & 28 & 199 & 1721 & 17394 \\
 r=1&1 & 5 & 28 & 199 & 1721 & 17394 & 200803 & 2607301 \\
 r=2&2 & 30 & 362 & 4390 & 56912 & 801668 & 12289342 & 204429498 \\
 r=3&6 & 180 & 3810 & 72960 & 1377936 & 26643204 & 536553870 & 11341749600 \\
 r=4&24 & 1200 & 39360 & 1099560 & 28812504 & 741799296 & 19236973920 & 509589280200 \\
 r=5&120 & 9000 & 422520 & 16237200 & 565687080 & 18805154760 & 614116782840 & 20053080534960 \\
 r=6&720 & 75600 & 4808160 & 243341280 & 10892100240 & 455188401360 & 18332566132320 & 725927285809440 \\
 \hline
\end{array}
$$$}
\end{table}

\vspace{2cc}


\bibliographystyle{amsplain}

\begin{thebibliography}{1}
    \bibitem{Aig}{\small {\sc M. Aigner:} {\it Combinatorial theory.} Springer Science \& Business Media, 2012.}
    \bibitem{Ap}{\small {\sc T. M. Apostol:} {\it Calculus, Volume I, One-variable Calculus, with an Introduction to Linear Algebra (Vol. 1).} John Wiley \& Sons, 2007.}
    \bibitem{Bon}{\small {\sc M. Bona:} {\it Combinatorics of Permutations (2nd ed.).} Chapman and Hall/CRC, 2012.}
    \bibitem{Bro}{\small {\sc A. Z. Broder:} {\it The $r$-Stirling numbers.} Discrete Math., {\bf 49(3)} (1984), 241--259.}
    \bibitem{Cai}{\small {\sc J. B. Caicedo, V. H. Moll, J. L. Ramírez, D. Villamizar:} {\it Extensions of set partitions and permutations.} Electron. J. Combin., {\bf 26(2)} (2019), P2--20.}
    \bibitem{Ch}{\small {\sc C. A. Charalambides:} {\it Enumerative combinatorics.} CRC Press, 2018.}
    \bibitem{Com}{\small {\sc L. Comtet:} {\it Advanced Combinatorics: The art of finite and infinite expansions.} Springer Science \& Business Media, 2012.}
    \bibitem{Dob}{\small {\sc G. Dobi\`nski:} {\it Summirung der Reihe $\sum \frac{n^m}{n!}$ für $m= 1, 2, 3, 4, 5, \dots$} Arch. für Mat. und Physik {\bf 61} (1877), 333--336.}
    \bibitem{Eps}{\small {\sc L. F. Epstein:} {\it A Function Related to the Series for $e^{e^{x}}$.} Stud. Appl. Math., {\bf 18(1-4)} (1939), 153--173.}
    \bibitem{Eul}{\small {\sc L. Euler:} {\it Institutiones calculi differentialis.} Teubner, 1755.}
    \bibitem{Fla}{\small {\sc P. Flajolet, R. Sedgewick:} {\it Analytic combinatorics.} cambridge University press, 2009.}
    \bibitem{Gro}{\small {\sc O. A. Gross:} {\it Preferential arrangements.} Amer. Math. Monthly, {\bf 69(1)} (1962), 4--8.}
    \bibitem{Lo}{\small {\sc N. Loehr:} {\it Combinatorics.} CRC Press, 2017.}
    \bibitem{Man}{\small {\sc T. Mansour, M. Schork:} {\it Commutation relations, normal ordering, and Stirling numbers.} CRC Press, 2015.}
    \bibitem{Mez}{\small {\sc I. Mez\H{o}:} {\it The $r$-Bell numbers.} J. Integer Seq., {\bf 14(1)} (2011), 1--14.}
    \bibitem{Qu}{\small {\sc J. Quaintance, H. W. Gould:} {\it Combinatorial identities for Stirling numbers: the unpublished notes of HW Gould.} World Scientific, 2015.}
    \bibitem{Rot}{\small {\sc G.-C. Rota:} {\it The number of partitions of a set.} Amer. Math. Monthly, {\bf 71(5)} (1964), 498--504.}
    \bibitem{Tan}{\small {\sc S. M. Tanny:} {\it On some numbers related to the Bell numbers.} Canad. Math. Bull., {\bf 17(5)} (1975), 733--738.}
    \bibitem{Wa}{\small {\sc C. Wang, P. Miska, I. Mez\H{o}:} {\it The $r$-derangement numbers.} Discrete Math., {\bf340(7)} (2017), 1681--1692.}
    \bibitem{Whi}{\small {\sc W. A. Whitworth:} {\it Choice and chance.} D. Bell and Company, 1870.}
    \bibitem{Wil}{\small {\sc H. S. Wilf:} {\it generatingfunctionology.} CRC press, 2005.}
    \end{thebibliography}

    \vspace{1cc}

{\small
\noindent
{\bf Hacène Belbachir}\\
USTHB, Faculty of Mathematics,\\
RECITS Laboratory,\\
BP 32, El Alia, 16111,\\
Bab Ezzouar, Algiers, Algeria.\\
E-mail: hacenebelbachir@gmail.com or hbelbachir@usthb.dz\\

\noindent
{\bf Yahia Djemmada}\\
USTHB, Faculty of Mathematics,\\
RECITS Laboratory,\\
BP 32, El Alia, 16111,\\
Bab Ezzouar, Algiers, Algeria.\\
E-mail: yahia.djem@gmail.com or ydjemmada@usthb.dz\\

\noindent
{\bf L\'aszl\'{o} N\'emeth}\\
University of Sopron,\\
Institute of Mathematics,\\ 
Sopron, Hungary.\\
and
associate member of USTHB, 
RECITS Laboratory.\\
E-mail: nemeth.laszlo@uni-sopron.hu
}

\end{document}